\documentclass{article}
\usepackage{biblatex}
\usepackage{rotating}

\usepackage{amsmath, amssymb,amscd,geometry,amsthm}
\usepackage{thmtools}
\usepackage{color}
\usepackage[usenames,dvipsnames,svgnames,table]{xcolor}
\usepackage{epsfig}
\usepackage{soul}
\setstcolor{red}
\sethlcolor{yellow}

\usepackage{enumitem}
\usepackage{cprotect}
\usepackage{tikz}
\usepackage{quiver}
\usepackage{hyperref}

\theoremstyle{definition}

\newtheorem{thm}{Theorem}[section]
\newtheorem{propo}{Proposition}[section]
\newtheorem{defn}{Definition}[section]

\newtheorem{lemma}{Lemma}[section]

\newtheorem{remark}{Remark}[section]

\renewcommand{\iff}{\leftrightarrow}

\renewcommand{\models}{\vDash}

\newcommand{\dom}{\text{dom}}
\newcommand{\rng}{\text{rng}}

\newcommand\LL{\mathcal{L}}

\newcommand\RR{\mathbb{R}}
\newcommand\K{\mathbb{K}}

\newcommand\Q{\mathbb{Q}}

\newcommand\PP{\mathbb{P}}

\newcommand\HH{\mathcal{H}}
\newcommand\MM{\mathcal{M}}
\newcommand\NN{\mathcal{N}}

\input{fitch.sty}
\numberwithin{equation}{subsection}
\usepackage{lipsum}
\bibliography{references}

\usepackage[utf8]{inputenc}
\title{The Combinatorics of Falsification and Hypothesis Testing}
\author{Reid Dale \\ \href{mailto:reiddale@berkeley.edu}{reiddale@berkeley.edu}}
\date{September 2022}

\begin{document}

\maketitle

\begin{abstract}
    The present paper is concerned with the question of how falsifiable a single proposition is in the short and long run. Formal Learning theorists such as Schulte and Juhl \autocite{schulte1996epistemology} have argued that long-run falsifiability is characterized by the topological notion of nowhere density in a suitable topological space. I argue that the short-run falsifiability of a hypothesis is in turn characterized by the VC finiteness of the hypothesis. Crucially, VC finite hypotheses correspond precisely to definable sets in NIP structures. I end the chapter by giving rigorous foundations for Mayo's \autocite{mayo2005evidence} conception of severe testing by way of a combinatorial, non-probabilistic notion of \textit{surprise}. VC finite hypotheses again appear as the hypotheses with guaranteed short-run surprise bounds. Therefore, NIP theories and VC finite hypotheses capture the notion of short-run falsifiability. 
\end{abstract}

\section{Introduction}

In this paper we begin by assessing the account of falsification given by Formal Learning Theorists Juhl and Schulte \autocite{schulte1996epistemology}. For them, a hypothesis $\HH$ is identified with a set of possible worlds. They consider the case where $\HH \subset 2^{\omega}$. For them, such a hypothesis is \textit{always falsifiable} provided that regardless of the results of some finite collection of observations, the hypothesis still has the potential to be falsified by some further collection of observational data. This, they show, is equivalent to the topological notion of the \textit{nowhere density} of $\HH$ inside $2^{\omega}$ equipped with the product topology. This account gives a good account of \textit{long-run} falsifiability, but fails to give a satisfactory account of \textit{short-run} falsification as there are no bounds on \textit{how long} it might take an agent to witness a crucial experiment. I define a \textit{sample} along a set $X$ is given by:

\begin{defn}
A \textit{sample} of $X$ is an injective function $f: \omega \to X$. A sample $f$ is \textit{full} provided $f$ is bijective. 
\end{defn}

Now, relative to a sample $f$ we define a notion of the \textit{surprise} of a hypothesis along the sample as follows:

\begin{defn}
Let $X$ be a set, $f:\omega \to X$ a sample of $X$, and $\HH\subset 2^X$ a hypothesis. The \textit{surprise} of $\HH$ is the function
\begin{equation*}
S(\HH,f,n) = 1 - \frac{|\HH{\upharpoonright}_{f([n])}|}{2^{|f([n])|}}.\qedhere
\end{equation*}

\end{defn}

The surprise of $\HH$ along the enumeration $f$ is the relative proportion of the states of the world incompatible with $\HH$. If $\HH$ is highly surprising, then it is compatible with only a small number of observations along the sample.

Very closely related to the NIP theories discussed in the static case of falsification, the VC finite hypothesis classes are characterized by the ability to obtain uniform bounds on surprise independent of sample.

Finally, we turn our attention to the work of Mayo \autocite{mayo2005evidence}\autocite{mayo2018severe}, who advocates for a strengthening of Null Hypothesis Statistical Testing as the foundation of statistical testing called \textit{severe testing}. Mayo and other error statisticians ask the question

\begin{quote}
When do data $x$ provide good evidence for / a good test of hypothesis $H$?
\end{quote}

The error statistician will invoke some form of a \textbf{Severity Principle} to answer this question:
\begin{quote}
(Weak Severity Principle) Data $x$ \textbf{does not} provide good evidence for $H$ if $x$ is the result of a \textbf{test procedure} $T$ with very low probability of uncovering the falsity of $H$\autocite[21]{mayo2010introduction}. 
\end{quote}

A converse is given by:
\begin{quote}
(Full Severity Principle) Data $x$ provides good evidence for $H$ to the extent that test $T$ has been \textbf{severely passed} by $H$\autocite[21]{mayo2010introduction}.
\end{quote}

However, the definition of severe testing via probabilistic notions is elusive. To round out the discussion of falsification, I show that the notion of surprise I defined in the context of always falsifiability is well-suited to give an account of a \textit{well-defined} combinatorial analogue of severe testing I term \textit{severe surprise}.

\begin{defn}
Let $f: \omega \to X$ be a sample, $\HH\subset 2^X$ a hypothesis, $n\in \omega$, and $\epsilon > 0$. We say that $(\HH,f,n)$ is \textit{severely surprising} at level $\epsilon$ provided the observed data $x\in \HH{\upharpoonright}_{f([n])}$,
\[ S(\HH,f,n) > 1-\epsilon\]
and
\begin{equation*}
 S(\HH,f,n) > S(\HH^c,f,n). \qedhere
\end{equation*}
\end{defn}

Crucially, surprise is at its core a non-probabilistic notion. Key to this is the observation that surprise is subadditive:

\begin{propo}
For all $\HH,f,n$,
\[S(\HH,f,n) + S(\HH^c,f,n) \leq  1 \qedhere\]
\end{propo}

and, in fact, any dense-codense pair $\HH, \HH^c \subset 2^{X}$ surprise 0 along any subsample. This is due to the simple fact that $\HH_1\cap \HH_2 = \varnothing$ does not entail that the restriction of the hypotheses to a finite set $X_0$, $\HH_1|_{X_0}$ and $\HH_2|_{X_0}$, are disjoint. 

Under this definition of severe surprise, one can show that VC finite classes are in fact severely surprising if true, uniformly in the size of the sample.

Thus, VC finiteness emerges as \textit{the} core notion of dynamic case of falsification, being robust under arbitrary samples and uniquely endowed with felicitous finite-sample bounds.

\subsection{The Formal Learning Theoretic Picture of Falsification}

The learning-theoretic analysis of falsification given by Schulte and Juhl identifies \textit{always falsifiability} of a hypothesis $\HH \subset 2^X$ with the notion of \textit{nowhere density} in the usual product topology on $2^X$.\autocite[10]{schulte1996epistemology}

The framework of Formal Learning Theory, especially as developed by Kelly \autocite{kelly1996logic}, provides a formal model of learning through observation. Schulte and Juhl \autocite{schulte1996epistemology} leverage this framework to give a topological characterization of Popperian falsification. On this model, an agent is idealized as being fed a countably infinite number of observations encoded by natural numbers and determining at each step $n$ whether or not some property $P$ holds of observation $x_n$: if $P(x_n)$, output $1$ and if $\neg P(x_n)$, output $0$. Such a sequence  is called a \textit{data stream}. Mathematically speaking, a data stream can be thought of simply as an element of Cantor space $2^{\omega}$. An \textit{empirical hypothesis} is simply a set of data streams, and thus a subset of $2^{\omega}$.

Popper's solution to the demarcation problem says that the distinguishing feature of a scientific theory--construed as an empirical hypothesis--is its \textit{falsifiability}. As Schulte and Juhl explain, a weak form of falsification is that under ideal circumstances the hypothesis can be conclusively ruled out on the basis of some observed relation or relations termed a \textit{crucial experiment} of the theory. However, if the crucial experiment does not rule out the theory, it may be the case that there are no other crucial experiments to run.

A more robust notion of falsifiability---what they term \textit{always falsifiability}---demands a preponderance of crucial experiments: given any finite collection of observations there exists a crucial experiment which may conceivably be run in the future. This notion expresses the idea that the scientific theory might never be \textit{confirmed} at any finite time, since there are always potential observational paths refuting it in the future. Schulte and Juhl demonstrate that the always falsifiability of a hypothesis $\mathcal{H} \subset 2^{\omega}$ is equivalent to that hypothesis' \textit{nowhere density} in the usual topology on $2^{\omega}$.

Within formal learning theory, a \textit{world} is an element of $2^{\omega}$ and a hypothesis $\HH$ is identified with the subset $\HH\subset 2^{\omega}$ of worlds in which $\HH$ holds. 

Given a hypothesis $\HH\subset 2^{\omega}$, we construct a two-sorted structure $\MM_{\HH}$ in the language $\LL = \{R(x,y),O(x),W(y)\}$ as follows:
\begin{enumerate}
    \item The domain of $\MM_{\HH}$ is the disjoint union $\omega \cup \HH$,
    \item The predicate $O$ consists of all of the \textit{observations}: $\MM_{\HH} \models O(x)$ just in case $x\in \omega$
    \item The predicate $W$ consists of all of the \textit{worlds}: $\MM_{\HH}\models W(y)$ just in case $y\in \HH$.
    \item $\MM_{\HH} \models R(x,y)$ just in case $\MM_{\HH}\models O(x) \wedge W(y)$,$x\in y$, and $y\in \HH$.
\end{enumerate}
In other words, the structure on $\MM_{\HH}$ is the bipartite graph on $\omega$ and $\HH$ with each world $w\in \HH$ encoding itself:
\[R(\MM_{H},w) = w.\]

It is worth noting that we can encode a lot of information into this framework. For instance, consider the structure $(\Q,<)$ with $<$ as the usual order on $\Q$. Then, enumerating $\Q^2$ as $\omega$, we may regard the partitioned formula \[\varphi(x;y_1,y_2) = x>y_1 \wedge x<y_2\] as a hypothesis $\HH_{\varphi}\subset 2^{\omega}$. We identify an $h\in \HH$ with any one of its codes; that is, $h = \varphi(x;h_1,h_2)$ for some $h_1,h_2 \in \Q$.
We then have a definable interpretation of the bipartite graph
\[\Q \cup \HH_{\varphi}(\Q) \]
with 
\[R(x,h) \iff x\in h \iff (x\in \Q \wedge ( \varphi(x;h_1,h_2)).\]

To make this fully model theoretic, the induced structure we study would not be restricted to only countable models; instead we would concern ourselves with large, sufficiently saturated $\MM \in \K$ and look at the induced bipartite structure with domain
\[\MM \cup \HH_{\varphi}.\]

We see that this is even necessary to capture all intervals in $\Q$ with real endpoints; restricting only to parameters $h_1,h_2\in \Q$ we have only countable many elements in $\HH_\varphi(\Q)$, but $\HH_{\varphi}(\RR)\cap \Q \subset 2^\Q$ is strictly larger.


The above construction gives us a way to convert learning questions in model theory with the setup of Formal Learning Theory. 

To maintain consistency with the Formal Learning Theory literature, we will work primarily in the standard setting of hypotheses $\HH\subset 2^{\omega}$, knowing that we may choose to encode mathematical structures into this framework as needed.

\subsection{Popper Dimension, VC dimension, and the Topology of Falsification}

The learning-theoretic analysis of falsification given by Schulte and Juhl identifies \textit{always falsifiability} of a hypothesis $\HH \subset 2^X$ with the notion of \textit{nowhere density} in the usual product topology on $2^X$.\autocite[10]{schulte1996epistemology}

There is an equivalent description of the notion of always falsifiability in terms of the fundamental machine-learning theoretic notion of shattering.

\begin{defn}
Let $\HH \subset 2^X$ be a hypothesis and $X_0 \subset X$. Then $\HH$ is said to shatter $X_0$ provided that the restriction of $\HH$ to $X_0$,
\[ \HH{\upharpoonright}_{X_0} = \{h{\upharpoonright}_{X_0}\,|\, h\in \HH \},   \]
satisfies
\[\HH{\upharpoonright}_{X_0} = 2^{X_0}.\qedhere\]
\end{defn}

One may give an equivalent definition of always falsifiability in terms of shattering, following the account of \autocite{seldin2013relations}.

\begin{defn} 
Let $\HH\subset 2^X$ be a hypothesis. Let $f: X_0 \to 2$ be a function defined on a finite subset $X_0 \subset X$. Given such a function, let $\HH_f = \{h\in \HH\,|\, h\supset f \}$ be the set of functions in $\HH$ extending $f$.

The Popper dimension $\delta_P$ of $\HH$ relative to $f$ is the size of the smallest subset of $X\setminus \dom(f)$ \textit{not} shattered by $\HH$. More precisely:
\[\delta_{P}(\HH,f) = \min\left\{|Y|\,\mid\, \text{$Y \subset (X\setminus \dom(f)) $ is not shattered by $\HH_f$ } \right\} .  \]
We say that $\HH$ is \textit{hereditarily Popper finite} provided $\delta_P(\HH,f)$ is finite for all $f:X_0\to \{0,1\}$ with finite domain.
\end{defn}
 
\begin{propo}
$\HH$ is always falsifiable if and only if $\HH$ is hereditarily Popper finite. 
\end{propo}

\begin{proof}
Suppose that $\HH$ is hereditarily Popper finite. To show that $\HH$ is nowhere dense, first suppose that $U\subset 2^{X}$ is a basic open set, say $U= U_s$ for some string $s$. We need to show that $\HH \cap U$ is not dense in $U$. It suffices to show that there exists a nonempty basic open $V\subset U$ such that $\HH \cap V = \varnothing$. Since $\HH$ is hereditarily Popper finite, there is a finite $n$ such that $\delta_P(\HH,s) = n < \infty$. Then there is a string $t \supset s$ of length $|s|+n+1$ such that $t \notin \HH_s$. Thus $\HH \cap U_t = \varnothing$. Since $U$ was arbitrary basic open, $\HH$ is nowhere dense.

Conversely, if $\HH$ is not hereditarily Popper finite then there exists a finite subset $X_0\subset X$ and $f:X_0\to \{0,1\}$ such that all finite subsets $Y_0 \subset X\setminus X_0$ are shattered by $\HH_f$. This precisely says that the nonempty basic open set $U_f$ is such that $\HH \cap U_f$ is dense. Thus $\HH$ is not nowhere dense.
\end{proof}

A stronger condition than hereditary Popper finiteness---the context of Vapnik's PAC learnability---is that of VC finite classes.

\begin{defn}
Let $\HH \subset 2^X$ be a hypothesis. The VC dimension of $\HH$ is the maximal size of a set shattered by $\HH$:
\[ \delta_{VC}(\HH) = \max\{|Y|\,\mid\, \text{$Y \subset X$ is shattered by $\HH$}  \}. \qedhere\]
\end{defn} 

\begin{propo}\label{prop:vc-popper}
\begin{enumerate}
    \item If $\HH$ is VC finite then $\HH$ is hereditarily Popper finite. 
    \item There exist hereditarily Popper finite $\HH$ which are not VC finite.
\end{enumerate}
\end{propo}
\begin{proof}
\begin{enumerate}
    \item This is essentially \autocite[Lemma 6.1]{seldin2013relations}. By definition, for all finite $x\subset X$ 
    \[\delta_P(\HH,x) \leq VC(\HH) + 1,\] so that if $\HH$ is VC finite then $\HH$ is hereditarily Popper finite. \qedhere
    
\item Let
\[\HH = \{f\in 2^{\omega}\,|\, (\forall n\in \omega \text{ even}) f(n) = 0 \}.\]
This set is hereditarily Popper finite as $\HH$ shatters \textit{no} set containing an even number $n$, but is VC infinite as $\HH$ shatters the collection of odd integers. \qedhere
\end{enumerate}
\end{proof}

Thus the VC finiteness of a hypothesis constitutes a stronger notion of falsifiability than that of always falsifiability. It turns out that the added constraints of VC finiteness are precisely what are needed to yield sample-independent bounds on the \textit{prevalence} of crucial experiments. 

\subsection{Surprise and Observational Studies} 

Over countable data streams, one can define a probability-independent notion of the \textit{surprise} of a hypothesis $\HH\subset 2^X$.

\begin{defn}
Let $X$ be countable. A \textit{sample} of $X$ is an injective function $f: \omega \to X$. A sample $f$ is \textit{full} provided $f$ is bijective.\end{defn}

\begin{defn}
Let $X$ be countable, $f:\omega \to X$ a sample of $X$, and $\HH\subset 2^X$ a hypothesis. The \textit{surprise} of $\HH$ is the function
\[ S(\HH,f,n) = 1 - \frac{|\HH{\upharpoonright}_{f([n])}|}{2^{|f([n])|}}. \qedhere\]
\end{defn}
The surprise of $\HH$ along the enumeration $f$ is the relative proportion of the states of the world incompatible with $\HH$. Surprise is a quantitative, probability-independent measure of falsifiability:

\begin{propo}
Suppose that $X$ is countable, $f:\omega\to X$ is a sample, and $\HH\subset 2^X$ is a hypothesis. Then there is a crucial experiment of $\HH$ along $f$ at stage $n$ if and only if $S(\HH,f,n) > 0$.
\end{propo}

\begin{proof}
By definition, a crucial experiment occurs just in case $|\HH{\upharpoonright}_{f([n])}| < 2^{|f([n])|}$, which is equivalent to saying that 
\[S(\HH,f,n) > 0. \qedhere\]
\end{proof}

In the case that $\HH$ is VC finite, the Sauer-Shelah lemma allows us to give uniform, enumeration-independent bounds on the surprise of $\HH$. The Sauer-Shelah lemma shows that the \textit{growth function} of a hypothesis class is polynomial once the sample size exceeds the VC dimension of the class:

\begin{lemma}\label{lem:sauer-shelah} Let $\HH$ be a hypothesis class of VC dimension $d$. Then for all $m$, the growth function
\[\tau_{\HH}(m) = \max\{|\HH{\upharpoonright}_{Y}|\,\mid\, \text{$Y\subset X$ and $|Y| = m$}  \} \]
satisfies the inequality
\[\tau_{\HH}(m) \leq \sum\limits_{i=0}^d {m\choose i} .\]
In particular, if $m > d+1$ then 
\[\tau_{\HH}(m) \leq \tau_{\HH}(m) \leq \left(\frac{em}{d}\right)^d.\qedhere\]
\end{lemma}

\begin{propo}\label{prop:short-wait}
Let $X$ be countable and $\HH\subset 2^X$ a VC finite hypothesis. Then for every $\epsilon > 0$ there is an $m > 0$ such that for all enumerations $f:X\to\omega$ 
\[S(\HH,f,m) \geq 1 - \epsilon.\]
Moreover, for all enumerations $f$
\[\lim\limits_{m\to\infty} S(\HH,f,m) = 1. \qedhere\]
\end{propo}
\begin{proof}
By the Sauer-Shelah lemma (\autoref{lem:sauer-shelah}), we have for $m > d+1$ the inequality
\[ S(\HH,f,m) \leq  \tau_{\HH}(m) \leq \left(\frac{em}{d}\right)^d \]
so that for all enumerations $f:X\to \omega$,
\[ S(\HH,f,m) \geq 1 - \frac{(em)^d}{2^m d^d}.\]
As $m\to\infty$, $\frac{(em)^d}{2^m e^d} \to 0$ and so
\[ \lim\limits_{m\to\infty} S(\HH,f,m) = 1.\qedhere\]
\end{proof}

On the other hand, in the case of a VC infinite class there exist samples on $X$ with surprise $0$ for unbounded time:

\begin{thm}\label{prop:long-wait}
Let $X$ be countable and $\HH \subset 2^X$ a VC infinite hypothesis. Then for every $m$ there exists a sample $f_m:\omega\to X$ such that for all $k<m$
\[ S(\HH,f_m,k) = 0 \qedhere\]
\end{thm}
\begin{proof}
Since $\HH$ is VC infinite there exists a set $X_m\subset X$ of size $m$ which is shattered. Let $f_m$ be any enumeration of $X$ enumerating $X_m$ first. Then for all $k<m$ 
\[S(\HH,f,k) = 1 - \frac{|\HH{\upharpoonright}_{f([k])}|}{2^k} = 1 - \frac{2^k}{2^k} = 0. \qedhere \]
\end{proof}

This result has the following epistemic interpretation: for an agent undertaking observational inquiry, VC infinite classes may take unboundedly long to yield nontrivial surprise. 
    
\subsection{Falsifiability and Control Studies}

In the previous section we saw how an agent in an impoverished epistemic state---only being able to conduct purely observational studies without any way to alter the data stream---is guaranteed short-run falsifiability of a hypothesis $\HH$ just in case the hypothesis is VC finite. In this section we characterize the falsifiability of a hypothesis $\HH$ in terms of the existence of certain \textit{selectors}---to be thought of as an agent's sequential choice of objects amongst those in $X$---witnessing crucial experiments. 

For example, consider a simplified account of a particle collision experiment wherein at each time $t$ the scientist observes the collision of two elementary particles and the output is recorded. If the hypothesis $H$ in question is a hypothesis concerning the result of a collision between bosons, then the scientist may have to wait an unboundedly long time witnessing irrelevant experiments (e.g. proton-proton collisions). To make the hypothesis \textit{efficiently} falsifiable requires some form of control over the sampling procedure. To this end we define the notion of a selector.

\begin{defn}
A selector $s:\omega \to X$ is an injective sample. \end{defn}

The always falsifiability of a hypothesis is a necessary and sufficient condition for the existence of efficiently falsifying the hypothesis with a selector.

\begin{propo}\label{prop:falsifiable-selector}
If $\HH$ is always falsifiable provided then there exists a selector $s:\omega \to X$ such that for each $m$, a crucial experiment will be performed by sampling $s([0,m+k])$ where $k = \delta_P(\HH,s[m])$.

Moreover, if for every string $x \in 2^{X_0}$ for $X_0 \subset X$ finite there is a selector $s$ such that $s([m]) = x$ and from $m+1$ onward satisfies that a crucial experiment will be performed by sampling $s([0,m+k])$ where $k = \delta_P(\HH,s[m])$. 
\end{propo}

\begin{proof}
We construct $s: \omega \to X$ in stages:
\begin{itemize}
    \item (Stage $n = 0$) Suppose $\delta_P(\HH,\varnothing) = k < \infty$
    Let $s(0),\dots,s(k-1)$ enumerate any set $Y\subset X$ witnessing $\delta_P(\HH,\varnothing) = k$.
    \item (Stage $n = m+1$) Suppose $\delta_P(\HH,s([m])) = k < \infty$ and suppose that $s$ is defined on range $[\ell]$. Then define
    $s(\ell+1),\dots,s(\ell+k)$ so as to enumerate any set $Y\subset X$ witnessing $\delta_P(\HH,s[m]) = k$.
\end{itemize}
By construction, $s$ is injective and defined on all of $\omega$, so $s$ is a selector.

The converse is immediate from the definition of hereditary Popper-finiteness.
\end{proof}

We interpret this result as saying that a hypothesis class is always falsifiable just in case an agent able to select data along a selector $s$ as in \autoref{prop:falsifiable-selector} can falsify $\HH$ with sample bounds given by the Popper dimensions $\delta_P(\HH,s([n]))$.

\subsection{VC Finiteness and Inferring Always Falsifiability on Subsamples}
In the setup considered above, we looked only at the always falsifiability of a hypothesis over a fixed, countable sample set $X$. However, in typical scientific inference we typically wish to probe a hypothesis $\mathcal{H}$ with the aid of some (possibly incomplete) sample of the world. In fact, we study hypotheses knowing full well that our sampling capabilities are bounded: we cannot directly perform tests in the ancient past or beyond the observable universe. 

This would be no issue if the inference 

\[\begin{nd}
\hypo {1}{\mathcal{H} \text{ is always falsifiable. }}
\have {2}{\mathcal{H}{\upharpoonright}_Y \text{ is always falsifiable.}} 
\end{nd}
\]

were true for all subsamples $Y\subset X$. However, this inference is invalid. Recall from topology that the interior of a set $X$, $\operatorname{int}(X)$, is the union of all open subsets $U\subseteq X$, and the closure of the set $X$, $\overline{X}$, is the intersection of all closed subsets $C\supseteq X$. Nowhere density of a set $X$ is typically defined by
\[ \operatorname{int}(\overline{\mathcal{H}}) = \varnothing.\]
In the case of \textit{finite} subsamples, recall that

\begin{propo}\label{prop:finite_nowhere_dense}
The only nowhere dense subset of $2^n$ is $\varnothing$.
\end{propo}
\begin{proof}
The product topology on $2^n$ is the discrete topology, so the interior and closure operators on subsets $\mathcal{H}\subset 2^n$ are equal to the identity operator. Thus, $\operatorname{int}(\overline{\mathcal{H}}) = \varnothing$ just in case $\mathcal{H} = \varnothing$. 
\end{proof}

From this observation it follows that
\begin{propo}
Suppose that $\mathcal{H} \subset 2^{\omega}$ is nonempty. Then for all $Y\subset X$ finite nonempty, $\mathcal{H}{\upharpoonright}_Y$ is not nowhere dense.
\end{propo}
\begin{proof}
Suppose that $\mathcal{H}$ is nonempty. Then $\mathcal{H}{\upharpoonright}_Y$ is nonempty, so cannot be nowhere dense by \autoref{prop:finite_nowhere_dense}.
\end{proof}

This result should not surprise us; after all, if there are only finitely many observations to be made then there will not always exist a further crucial test to perform as required by always falsifiability. 

Less trivially, we can exhibit the existence of nowhere dense hypotheses $\mathcal{H}$ such that the inference 
\[\begin{nd}
\hypo {1}{\mathcal{H} \text{ is always falsifiable. }}
\have {2}{\mathcal{H}{\upharpoonright}_{Y} \text{ is always falsifiable.}} 
\end{nd}
\]
fails on a continuum-sized ideal of samples $Y \subset 2^{\omega}$.

\begin{propo}
Let $X\subset \omega$ be infinite-coinfinite. Then there exists a nowhere dense $\mathcal{H}_X\subset 2^{\omega}$ such that for all $Y\subseteq X$, $\mathcal{H}_X{\upharpoonright}_Y = 2^Y$ and therefore not nowhere dense.
\end{propo}
\begin{proof}
Define $\mathcal{H}_X$ as the set of all functions $f:\omega \to \{0,1\}$ such that $f(x) = 0$ if $x\notin X$. 

Because $X$ is coinfinite, the set $\mathcal{H}_X$ is nowhere dense: every $x\notin X$ yields a crucial experiment. Moreover, $\mathcal{H}_X{\upharpoonright}_X = 2^X$ by definition, and likewise for any $Y\subseteq X$ we have that $\mathcal{H}_X{\upharpoonright}_Y = 2^Y$.
\end{proof}

While nowhere dense over the full sample set $\omega$, the hypotheses $\mathcal{H}_X$ fail to be always falsifiable on any $Y\subseteq X$.

On the other hand, the stronger notion of VC finiteness is preserved under the implication above:

\begin{propo}
For all $Y\subseteq X$ the inference rule
\[\begin{nd}
\hypo {1}{\mathcal{H} \text{ is VC finite.}}
\have {2}{\mathcal{H}{\upharpoonright}_Y \text{ is VC finite.}} 
\end{nd}
\]
is valid.
\end{propo}
\begin{proof}
Observe that if $Y_0 \subset Y$ is shattered by $\HH{\upharpoonright}_Y$ then $Y_0$ is itself shattered by $\HH$. Thus
\[VC(\HH{\upharpoonright}_Y) \leq VC(\HH) \]
so $\HH\cap 2^Y$ is VC finite.
\end{proof}

VC finiteness does not, however, characterize those hypothesis classes which are hereditarily nowhere dense.

\begin{propo}
Partition $\omega = \bigcup\limits_{n\in\omega} X_n$ where each $|X_n| = n$. Let $\HH = \bigcup\limits_{n\in\omega} 2^{X_n}$. Then
\begin{enumerate}
    \item for every infinite $Y\subset \omega$, $\HH{\upharpoonright}_Y$ is nowhere dense in $2^Y$, and
    \item $\HH$ is VC infinite. \qedhere
\end{enumerate}
\end{propo}
\begin{proof}

By construction, $\HH$ shatters only finite sets, so is hereditarily Popper finite.

$\HH$ is VC infinite since, by construction, it shatters arbitrarily large sets.
\end{proof}

Despite this, there is a precise sense in which one can say that if $\HH\subset 2^{\omega}$ is a hypothesis of infinite VC dimension then the structure $\MM_{\HH}$ is observationally indistinguishable from a structure $\NN$ such that that an infinite set is shattered by the relation $R$.

\begin{propo}
Let $\HH\subset 2^{\omega}$ be VC infinite. Then there exists an $\NN$ elementarily equivalent to $\MM_{\HH}$ such that the interpretation of $\HH$ in $\NN$, $\HH^{*}$, shatters an infinite set.
\end{propo}
\begin{proof}
Let $\NN$ be a sufficiently saturated nonprincipal ultrapower of $\MM_{\HH}$. Then 
\[ \NN = \omega^{*} \cup \HH^{*} \]
defines the structure of a hypothesis set on $2^{\omega*}$. $\HH^{*}$ is regarded as a subset of $2^{\omega^{*}}$ by way of the interpretation of the relation $R$. That is, we may regard
\[ \NN \models R(n^{*},h^{*}) \iff n^{*}\in h^{*}\]
as the definition of an embedding $\HH^{*} \subset 2^{\omega^{*}}$.

By saturation, $\NN$ shatters an infinite set as $\HH$ shatters arbitrarily large finite sets.
\end{proof} 

We note here that the $\NN$ as constructed in the above proposition is \textit{elementarily equivalent} to $\MM_{\HH}$. This suffices to conclude that, in a strong sense, no finitistic agent will ever be able to discern between $\MM_{\HH}$ and $\NN$. This is due to the equivalence between elementary equivalence and finitary back-and-forth equivalence.

\begin{defn}\autocite[Definition XI.1.1]{ebbinghaus2013ml} Let $\MM$ and $\NN$ be $\LL$-structures. A partial function $f:\MM \dashrightarrow \NN$ with domain $\dom(f)\subseteq \MM$ and range $\rng(f)\subseteq \NN$ is a \textit{partial isomorphism} provided 
\begin{enumerate}
    \item $f$ is injective,
    \item $f$ preserves all relations, function symbols, and constants in $\LL$. \qedhere
\end{enumerate}

\end{defn}

Finitary back-and-forth equivalence is a property about being able to extend arbitrary partial isomorphisms with finite domain:

\begin{defn}\autocite[Definition XI.1.3]{ebbinghaus2013ml}
Two $\LL$-structures $\MM$ and $\NN$ are finitarily back-and-forth equivalence provided there is a sequence $(I_n)_{n\in\omega}$ such that
\begin{itemize}
    \item Every $I_n$ is a nonempty set of partial isomorphisms from $\MM$ to $\NN$,
    \item (Forth) For every $f\in I_{n+1}$ and $a\in \MM$ there is a $g\in I_{n}$ with $g\supseteq f$ and $a\in \dom(g)$
    \item (Back) For every $f\in I_{n+1}$ and $b\in \NN$ there is a $g\in I_{n}$ with $g\supseteq f$ and $a\in \rng(g)$. \qedhere
\end{itemize}
\end{defn}

The definition of finitary back-and-forth equivalence has an immediate epistemic interpretation. Two structures being back-and-forth equivalent means that any finite quantifier-free relation in $\MM$ can be witnessed in $\NN$ and vice versa. Thus, no finite amount of observation of quantifier-free formulas can discern between $\MM$ and $\NN$.\footnote{The astute reading will note that the definition of back-and-forth equivalence requires partial isomorphism between \textit{finitely-generated substructures,} which in the case of a language with function symbols may be infinite. One may remedy this by nothing that any theory $T$ in a language $\LL$ containing constant and function symbols is bi{\"i}nterpretable with a theory $T'$ in a purely relational language, where the finitely generated structures in a relational language are precisely the finite structures.} Fraiss{\'e}'s theorem relates finitary back-and-forth equivalence with elementary equivalence:

\begin{thm}\autocite[Theorem XI.2.1]{ebbinghaus2013ml}\label{thm:back_and_forth}
Let $\LL$ be a finite language. Two $\LL$-structures $\MM$ and $\NN$ are finitely back-and-forth equivalent if and only if they are elementarily equivalent
\end{thm}

Thus, even if $\HH$ \textit{happens} to be nowhere dense, VC infinite yet does not shatter an infinite set, in a strong sense $\HH$ is observationally indistinguishable from one in which which \textit{does} shatter an infinite set. 

This result illustrates an effect of the underlying framework of Formal Learning Theory: it works assuming an agent knows the \textit{extensional specification} of the space of observations ($\omega$) and hypotheses ($\HH$) on the nose.  However, bounded agents may only grasp the domain of observations and hypotheses intensionally, and thus know the hypothesis and sample domain only up to back-and-forth equivalence.

Viewed in this light, the VC finite classes emerge as precisely the class of nowhere dense hypotheses invariant under observable indistinguishability by finitistically bounded agents. 

\subsection{VC Finiteness is Not a Topological Notion}

In this section we argue that the topological and descriptive set theoretic tools relied upon in formal learning theory are too coarse to adequately study the short-run properties of the hypotheses of the sort encountered in machine learning.

A unifying theme of theoretical machine learning is identifying combinatorial notions of dimension on hypotheses such that "finite dimensional iff learnable" is true. These notions of dimensions standardly have the structure of a \textit{nontrivial set-theoretic ideal} on $2^X$ in the case that $X$ is an infinite set. Two examples of such dimensions are \textit{VC dimension}, characterizing the PAC learnable hypotheses, and \textit{Littlestone dimension}, characterizing the online-learnable hypotheses. 

Following the analysis of \autocite{ciesielski1995topologies} we investigate the topologies arising from such ideals and conclude that the natural topologies fail to satisfy the standard metrization requirements of Formal Learning Theory. Instead, combinatorial measures of hypotheses are better equipped to handle such questions. 

To illustrate this general point, we see that the class of VC finite hypotheses cannot be realized as the nowhere dense sets in a Hausdorff topological space. The arguments here are drawn from the analysis of topological properties of set-theoretic ideals given by Cieselski and Jasinski in \autocite{ciesielski1995topologies}.

\begin{propo}
The set of VC finite families on an infinite set $X$,
\[I_{VC}(X) = \{Y\,|\, \text{$Y\subset X$ and $VC(Y) < \infty$} \}\]
forms a proper ideal in $2^{2^X}$.
\end{propo}
\begin{proof}
First, $2^X\notin I_{VC}(X)$ since, by definition, $2^X$ shatters an infinite set. Moreover, it is clear from the definitions that if $Y \in I_{VC}$ and $Z\subset Y$ then $Z\in I_{VC}$ since every set shattered by $Z$ is shattered by $Y$. 

Finally, the Sauer-Shelah lemma implies that if $Y,Z\in I_{VC}$ then $Y\cup Z \in I_{VC}$ since the growth function of the union is polynomial.
\end{proof} 

The fact that $I_{VC}$ has the structure of a proper ideal on $2^X$ means that we may construct a topology in which the VC finite sets are precisely the closed sets. However, this topology is non-Hausdorff. 

\begin{propo}
The collection
\[\tau(2^X) = \{2^X\setminus \HH\,|\, \HH\in I_{VC}(X) \} \cup \{\varnothing\} \]
of subsets of $2^X$ forms a non-Hausdorff topology on $2^X$. 
\end{propo}
\begin{proof}
Since $I_{VC}(X)$ has the structure of a set-theoretic ideal, the collection
\[F_{VC}(X) = \{2^X\setminus \HH\,|\, \HH\in I_{VC}(X) \} \]
is a nontrivial filter on $X$. Thus, the collection
\[ \tau(2^X) = F_{VC}(X) \cup \{\varnothing\} \]
is closed under arbitrary union, finite intersection, and contains $\varnothing$ and $2^X$. 

This topology is non-Hausdorff: for any $U,V \in \tau(2^X)$,
\[(U\cap V = \varnothing) \rightarrow (U=\varnothing \vee V = \varnothing)\]
since $\varnothing \notin F_{VC}(X)$ and $F_{VC}(X)$ is closed under finite intersection.
\end{proof}

Moreover, no Polish space can make all VC finite sets closed.

\begin{propo}
Let $X$ be countably infinite. Then there are $2^{2^{\aleph_0}}$ VC finite subsets of $X$. In particular, no Polish topology renders all VC finite sets closed.
\end{propo}
\begin{proof}
Since $I_{VC}$ is closed downward, it suffices to show that there is an uncountable VC finite subset of $2^X$.

Identifying $X$ with $\Q$, we may identify the family of intervals $\HH = \{ (r,\infty)\,|\, r\in \RR\}$ with a VC finite subset of $2^X$. This family has size $2^{\aleph_0}$, so $|I_{VC}| = 2^{2^{\aleph_0}}$.

Since Polish spaces have at most $2^{\aleph_0}$ many closed sets, no Polish space renders all VC finite subsets closed.
\end{proof}

Finally, we identify a mild condition on topologies guaranteeing that the ideal of nowhere dense sets does not coincide with the ideal of VC finite sets.

\begin{thm}\label{thm:comb-not-top} Let $X$ be infinite. There is no topology $\tau$ on $2^X$ such that
\begin{enumerate}
    \item There is a countable disjoint collection of nonempty open sets $U_n$ such that each $U_n$ shatters a set of size $\geq n$,
    \item Every VC finite $\HH$ is nowhere dense, and
    \item $I_{VC}(X) = I_{nd}(X)$. \qedhere
\end{enumerate}

\end{thm}
\begin{proof} 

We follow the proof strategy outlined in \autocite[Thm 3.4]{ciesielski1995topologies} by showing that if $\tau$ were a topology on $X$ making all VC finite sets nowhere dense then there is a VC infinite nowhere dense subset. 


By hypothesis 1, 
there exists a countably infinite set of disjoint open subsets $U_n$ shattering a set of size $\geq n$. 

Let $\HH_n\subset U_n$ be a finite hypothesis class shattering a set of size $n$. Then the hypothesis $\HH = \bigcup\limits_{n\in\omega} \HH_n$ shatters arbitrarily large subsets by construction, and is nowhere dense as each $\HH_n$ is finite and concentrated on a single open set $U_n$.
\end{proof}


In particular, the standard results of descriptive set theory---requiring that the topology in question be Polish and hence Hausdorff---do not apply to \textit{any} topology rendering the learnable sets nowhere dense.
 
\section{Rigorous Foundations for Severe Testing}

Null Hypothesis Statistical Testing (NHST) is a ubiquitous method of statistical inference. As Wasserman \autocite[Chapter 10]{wasserman2006all} describes it, the basic data of a Null Hypothesis Statistical Testing consists of
\begin{enumerate}
    \item A space $\Theta$ of probability distributions on sample space $\Omega$,
    \item A partition $\Theta = \HH_0 \cup \HH_1$
    \item A random variable $T:\Omega \to \RR$ called the test statistic, 
    \item A critical value $c\in \RR$.
\end{enumerate}
In the setup of a two-sided test, it is assumed that the null hypothesis $\HH_0$ is a \textit{single distribution}, i.e. $\HH_0 = \{\theta_0\}$, and therefore unambiguously determines a probability measure $\PP_{\HH_0}$ that we may use to infer probability statements about the test statistic $T$. To \textit{reject} hypothesis $\HH_0$, a statistical version of modus tollens is invoked, by replacing $\HH_0\rightarrow (T(x) \leq c)$ with $\PP_{\mathcal{H}_0}(T(x) > c) < \epsilon$:
\[\begin{nd}
\hypo {1}{\PP_{\mathcal{H}_0}(T(x) > c) < \epsilon}
\hypo {2}{T(x)>c}
\have{3}{Rej(\HH_0)}
\end{nd}.
\]

Mayo and other error statisticians instead advocate for a modern recasting of NHST---severe testing---as the appropriate framework guiding the use of statistical methods.  Central to the error statistician is the question:

\begin{quote}
When do data $x$ provide good evidence for/a good test of hypothesis $H$?
\end{quote}

The error statistician will invoke some form of a \textbf{Severity Principle} to answer this question:
\begin{quote}
(Weak Severity Principle) Data $x$ \textbf{does not} provide good evidence for $H$ if $x$ is the result of a \textbf{test procedure} $T$ with very low probability of uncovering the falsity of $H$ \autocite[21]{mayo2010introduction}. 
\end{quote}

A converse is given by:
\begin{quote}
(Full Severity Principle) Data $x$ provides good evidence for $H$ to the extent that test $T$ has been \textbf{severely passed} by $H$ \autocite[21]{mayo2010introduction}.
\end{quote}

The error statistician naturally asks \textit{which} hypotheses are amenable to error-theoretic analysis. This question is of utmost importance as the Full Severity Principle suggests the following account of scientific content: the hypotheses $H$ that have scientific content are precisely those which are severely testable. But what is the definition of severe testing?

The notion of severe testing as described by Mayo is defined as follows:
\begin{defn}
A hypothesis $H$ passes a severe test relative to experiment $E$ with data $x$ if (and only if):
\begin{enumerate}[label=\roman*]
\item $x$ agrees with or ``fits'' $H$ (for a suitable notion of fit), and
\item experiment $E$ would (with very high probability) have produced a result that fits $H$ less well than $x$ does, if $H$ were false or incorrect. \autocite[99]{mayo2005evidence} \qedhere
\end{enumerate}
\end{defn}

We turn now to discussing prongs $(i)$ and $(ii)$ in the above definition.

Regarding $(i)$, Mayo writes that ``fit'' should at the very least be 
\begin{equation*} \label{eq:mayo_fit}
\PP(x;H) > \PP(x; \neg H)
\end{equation*}
arguing that 
\begin{quote} 
any measure of evidential relationship, degree of confirmation, probability, etc., can be regarded as supplying a fit measure. Severity can then be assessed by computing the error probability required in (ii). \autocite[124]{mayo2005evidence}
\end{quote} 
If the notation $\PP(x;H)$ is unfamiliar, that is for good reason: Mayo explains that
\begin{quote}
I am using ``;'' in writing $\PP(x;H)$---in contrast to the notation typically used for a conditional probability, $\PP(x|H)$---in order to emphasize that severity does \textit{not} use a conditional probability which, strictly speaking, requires that the prior probabilities $\PP(H_i)$ be well-defined for an exhaustive set of hypotheses. \autocite[102]{mayo2005evidence}
\end{quote} 
This is on the face of it a key departure from the framework of NHST, which requires us to \textit{only} work with probability sentences involving $\PP_{\HH_0}$.
A serious difficulty for this account of severe testing is that no general construction of $\PP(x;\neg\HH)$ is given. 

For example, let $\HH_{r}\subset 2^{\omega}$ be the  a statement such as ``the long run relative frequency of heads in a countable sequence coin flips is equal to $r$.'' The complement $\HH_{r}^c \subset 2^{\omega}$ can be decomposed as the disjoint union
\[\HH_r^c = \HH_{\uparrow} \cup \bigcup\limits_{s\neq r} \HH_s \]
where $\HH_{\uparrow}$ is the set of all infinite binary strings with non-convergent limiting relative frequency as well as $\HH_{r'}$ for all $r' \neq r$. Moreover, $\HH_{\uparrow}$ and $\HH_{s}$ are dense and codense in $2^{\omega}$, so the complement $\HH_r^c$ has rich topological structure. There is no clear way to construct a probability measure that amalgamates all $\HH^{c}$ into a probability measure $\PP_H(x;\HH^c)$ if one does not avail oneself to an aggregation function such as a Bayesian prior. 

Even restricting only to the probability distributions on $\{0,1\}$, which we identify with the interval $[0,1]$, non-Bayesian methods of aggregating families of probability distributions---such as the Maximum Likelihood Estimator---are generally not well-defined. Recall the definition of the Maximum Likelihood Estimator \autocite[Definition 9.7]{wasserman2006all} 
\begin{defn}
Let $\Theta$ be a family of distributions over $\Omega$ and $X_1,\dots, X_n:\omega \to \RR$ be an IID set of random variables. The likelihood function is given by
\[\LL_n(\theta) = \prod\limits_{i=1}^{n} f(X_i;\theta) \]
where the $f(X_i;\theta)$ are the probability density functions of the random variables $X_i$ with respect to distribution $\theta$.
\end{defn}

The maximum likelihood estimator is typically defined as \textit{the} value $\operatorname{MLE}(\theta,n) \in \Theta$ maximizing $\LL_n(\theta)$. However, this definition is misleading, as $\operatorname{MLE}(\theta,n)$ may fail to exist or to be unique.

First, the Maximum Likelihood Estimator may fail to be unique. Let $\Theta = \{0,1\}$ be the space of distributions asserting that all flips of a coin are heads or tails. Confronted with observations $\overline{x} = (H,T)$, the likelihood functions have values
\[\LL_2(0) = 0 = \LL_2(1)\]
and so \textit{both} $0$ and $1$ maximize the likelihood function relative to $\theta$.

Second, the Maximum Likelihood Estimator may fail to exist within $\Theta$. Let $\HH_0 = \{\frac{1}{2}\}$ and $\HH_1 = \HH_0^c = [0,\frac{1}{2})\cup(\frac{1}{2},1]$. Suppose that $\overline{x} = (H,T)$. A routine calculation \autocite[Example 9.10]{wasserman2006all}
\[\LL_n(\theta) = \prod\limits_{i=1}^n p_H^{X_i}(1-p_H)^{1-X_i} = p_H^S \times (1-p_H)^{n-S} \]
where $S$ is the number of heads in the sequence of coin flips. Taking the derivative of $\LL_n(\theta)$ and setting it to zero we find that $\theta = \frac{1}{2}$ is the unique maximum likelihood estimator on $[0,1]$ with likelihood $\frac{1}{4}$. While $\frac{1}{2}\notin \HH_1$, this does not on its own show that the Maximum Likelihood Estimator does not exist in $\HH_1$. In fact, \textit{any} paritition of a compact connected space $\Theta = \HH_1\cup\HH_2$ into nonempty subfamilies of distributions will suffer this defect since the existence of a $\theta\in\HH_i$ maximizing likelihood is guaranteed only if $\HH_i$ is closed. However, it is not hard to see that the image $\LL_n(\HH_1) = [0,\frac{1}{4})$. In other words, the likelihood function on $\HH_1$ is arbitrarily close to $\frac{1}{4}$ but never obtains that value. So, no maximum likelihood estimator exists on $\HH_1$. Thus, the standard frequentist method of aggregating probability distributions in light of data is not even generally well-defined, and cannot serve as a definition of $\PP(x;\neg\HH)$.

Regarding condition (ii) in the definition of severe testing, Mayo requires the satisfaction of the following conditional: if $\HH$ is false, then $E$ would have produced a result that fits $\HH$ less well than $x$ does with high probability. In this conditional, we assume that $\HH$ is false and tasked with computing \textit{some} probability given $\neg \HH$ and the specification of the experiment $E$. This poses a serious problem for her account of severe testing; she gives no general theory of semantics for the probabilistic statements comprising the definition of severe, as no method for determining how to construct a probability distribution $\PP_{E,\neg\HH}$ from experiment $E$ is given. 
With no way to determine what a ``good'' probability distribution is, it is difficult to make sense of this.

For instance, consider the case of the hypothesis $\HH_0$ expressing ``all flips of a coin are heads.'' Let $E_n$ be the experiment given by flipping the coin some very large number $n$ of times. Then, given such an experiment $E$, 

\begin{propo}
Let $\epsilon > 0$. There exists a probability distribution $\PP_n$ on $\{H,T\}$ such that the probability $p_H$ of flipping at least one tails $T$ on $n$ IID flips is $<\epsilon$. 
\end{propo}
\begin{proof}
Let $\PP$ be a probability distribution on  $\{H,T\}$. Let $p_H$ be the probability of heads. The probability $p$ of flipping at least one tails $T$ on $n$ i.i.d. flips is
\[ \PP(\text{At least one Tails}) = 1- \PP(\text{All Heads}) =  1-p_H^n. \]
Let $\epsilon > 0$. Then 
\[ \PP(\text{At least one Tails}) = 1-p_H^n < \epsilon \]
is equivalent to saying
\[p_H > (1-\epsilon)^{\frac{1}{n}}. \qedhere\]
\end{proof}

The upshot is that the specification of the experimental setup itself does not determine \textit{a priori} the relevant probability distribution is. Moreover, this hypothesis is as falsifiable as it can be: for each coin flip, $\HH$ is compatible with only a single outcome. Yet, on a strict reading of Mayo's definition, $\HH$ cannot be severely tested. 

Rather, it seems to me that the combinatorics of the hypothesis---not any notion of probability---are what make $\HH$ severely testable. The critical element of Mayo's definition of severe testing is that the data $x$ be compatible with $\HH$ and---simultaneously---highly incompatible with $\neg \HH$. That is, for a hypothesis $\HH$ to be severely tested by an experiment with $n$ observations $x$, we would require:
\begin{enumerate}
    \item $x\in \HH{\upharpoonright}_{[n]}$, and
    \item $\HH{\upharpoonright}_{[n]} \ll \HH^c{\upharpoonright}_{[n]}$.
\end{enumerate}
These requirements are precisely captured by the previously-defined notion of \textit{surprise}: if $\HH{\upharpoonright}_{[n]} \ll \HH^c{\upharpoonright}_{[n]}$ then $\frac{\HH{\upharpoonright}_{[n]}}{2^n} < \frac{\HH{\upharpoonright}_{[n]}}{\HH^c{\upharpoonright}_{[n]}}  \approx 0$. This motivates the following definition:

\begin{defn}
Let $f: \omega \to X$ be a sample, $\HH\subset 2^X$ a hypothesis, $n\in \omega$, and $\epsilon > 0$. We say that $(\HH,f,n)$ is \textit{severely surprising} at level $\epsilon$ provided the observed data $x\in \HH{\upharpoonright}_{f([n])}$,
\[ S(\HH,f,n) > 1-\epsilon\]
and
\[ S(\HH,f,n) > S(\HH^c,f,n). \qedhere \]
\end{defn}

Taking a step back, it is worth relating this definition back to probability theory. Crucially, the definition of surprise is non-probabilistic in the sense that  there probabilistic \textit{frequency semi-measure} lurking in the background:

\begin{defn}
Let $f:\omega\to X$ be a sample and $n\in \omega$. For every $\HH  \in 2^{\omega}$ $\mu_{f,n}:2^{\omega} \to [0,1]$ as follows:
\[ \mu_{f,n}(\HH) = \frac{|\HH{\upharpoonright}_{f([n])}|}{2^n}.\qedhere\]
\end{defn}

\begin{propo}
Let $f:\omega\to X$ be a sample and $n\in \omega$. Then $\mu = \mu_{f,n}$ is a bounded sub-additive function on $2^{2^{\omega}}$. More precisely, for $\HH_1,\HH_2\in 2^{\omega}$ we have
\[\mu(\HH_1\cup \HH_2) \leq \mu(\HH_1) + \mu(\HH_2).\]
Moreover, a necessary and sufficient condition for the inequality above to be strict for $\HH_1, \HH_2\in 2^{\omega}$ is that 
\[ \HH_1{\upharpoonright}_{f([n])} \cap \HH_2{\upharpoonright}_{f([n])} \neq \varnothing. \]
Finally, every dense-codense $\mathcal{D}\subset 2^{\omega}$ satisfies $\mu(\mathcal{D}) = \mu(\mathcal{D}^c) = 1$ so that 
\[\mu(\mathcal{D}\cup \mathcal{D}^c) = 1 < 2 = \mu(\mathcal{D}) + \mu(\mathcal{D}^c). \qedhere \]
\end{propo}
\begin{proof} 
Expanding terms, we find that
\begin{equation*}
\begin{split}
\mu(\HH_1) + \mu(\HH_2) & =  \frac{|\HH_1{\upharpoonright}_{f([n])}|+|\HH_2{\upharpoonright}_{f([n])}|}{2^n} \\
    & = \frac{|(\HH_1{{\upharpoonright}_{f([n])}}\setminus\HH_2{\upharpoonright}_{f([n])})| + 2|(\HH_1{{\upharpoonright}_{f([n])}} \cap \HH_2{\upharpoonright}_{f([n])})| + |(\HH_2|_{f([n])}\setminus \HH_1{\upharpoonright}_{f([n])})|}{2^n} \\
    & \geq   \frac{|(\HH_1{\upharpoonright}_{f([n])}\setminus\HH_2{\upharpoonright}_{f([n])})| + |(\HH_1{\upharpoonright}_{f([n])} \cap \HH_2{\upharpoonright}_{f([n])})| + |(\HH_2{\upharpoonright}_{f([n])}\setminus \HH_1{\upharpoonright}_{f([n])})|}{2^n} \\
    & = \mu(\HH_1 \cup \HH_2)
\end{split}
\end{equation*}
By the above chain of equalities, we see that 
\[\mu(\HH_1)+\mu(\HH_2) - \mu(\HH_1\cup\HH_2) = \frac{|(\HH_1{\upharpoonright}_{f([n])} \cap \HH_2{\upharpoonright}_{f([n])})|}{2^n}.\]
Now, while $\HH_1\cap \HH_2 = \varnothing$, that does not imply that $(\HH_1{\upharpoonright}_{f([n])} \cap \HH_2{\upharpoonright}_{f([n])}) = \varnothing.$

Finally, suppose that $\mathcal{D}$ is dense-codense. Then $\mathcal{D}|_{f([n])} = 2^{f([n])} = \mathcal{D}^c|_{f([n])}$, so 
\[\mu(\mathcal{D}) =\mu(\mathcal{D}) = \frac{|2^{f([n])}|}{|2^{f([n])}|} = 1. \qedhere\]
\end{proof}

The relationship between $\mu_{f,n}$ and the surprise function $S(-,f,n)$ is that 
\[S(\HH,f,n) = 1 - \mu_{f,n}(\HH).\]
Therefore

\begin{propo}\label{prop:surprise_ineq}
For all $\HH,f,n$,
\[S(\HH,f,n) + S(\HH^c,f,n) \leq  1\qedhere\]
\end{propo}
\begin{proof}
Immediate from the sub-additivity of $\mu_{f,n}$
\end{proof}

\begin{remark}
While the surprise function $S(\HH,f,n)$ has a direct epistemic interpretation, the above  the \textit{co-surprise function} has structure reminiscent of a probability measure. Let $S^{\operatorname{co}}(\HH,f,n) = S(\HH^c,f,n)$. Then
\begin{enumerate}
    \item $S^{\operatorname{co}}(\varnothing,f,n) = S(2^{\omega},f,n) = 0$,
    \item $S^{\operatorname{co}}(2^{\omega},f,n) = S(\varnothing,f,n) = 1$,
    \item $S^{\operatorname{co}}(\HH,f,n) + S^{\operatorname{co}}(\HH^c,f,n) = S(\HH^c,f,n) + S(\HH,f,n) \leq  1 $, and
    \item if $\HH_1 \subset \HH_2$ then $S^{\operatorname{co}}(\HH_1,f,n) \leq S^{\operatorname{co}}(\HH_2,f,n)$ as $\HH_1^c \supseteq \HH_2^c$.
\end{enumerate}
Co-surprise fails to be finitely additive, as any dense-codense subset $\mathcal{D}$ of $2^{\omega}$ has $S^{\operatorname{co}}(\mathcal{D},f,n) = 0 = S^{\operatorname{co}}(\mathcal{D}^c,f,n)$ as \[\mathcal{D}{\upharpoonright}_{f([n])} = 2^{f([n])} = \mathcal{D}^c{\upharpoonright}_{f([n])} ,\] so is not a measure.

In fact, this function is very far from being a measure in the sense that the maximal Boolean subalgebra of $2^{2^X}$ on which $S^{co}(-,f,n)$ is a finitely additive probability measure is rather small. For each $s\in 2^{f([n])}$, let 
\[ \mathcal{J}_s = \{h\in 2^X\,|\, h\supseteq s\}.\]
It is clear that 
\[ \HH{\upharpoonright}_{f([n])} \cap \HH^c{\upharpoonright}_{f([n])} = \varnothing \]
just in case for some $S\subseteq f([n])$
\[ \HH = \bigcup\limits_{s\in S} \mathcal{J}_s,\]
for if $\mathcal{J}_s \cap \HH \neq \varnothing$ and $\mathcal{J}_s\cap \HH^c \neq \varnothing$ then 
\[\HH{\upharpoonright}_{f([n])} \cap \HH^c{\upharpoonright}_{f([n])} \supseteq \{s\} \neq \varnothing.\] From this observation it is we may conclude that the Boolean algebra
\[\mathcal{A}_{f,n} = \{\mathcal{J}_s\,|\, s\in 2^{f([n])} \} \subseteq 2^{2^X} \]
is \textit{the} maximal Boolean algebra on which $S^{co}$ is a probability measure. 

This Boolean algebra is naturally isomorphic to $2^{f([n])}$ generated by the assignment $\mathcal{J}_s \mapsto s$. Under this identification, $S^{co}(-,f,n)$ coincides with the uniform measure on $2^{f([n])}$:
\[S^{co}(\mathcal{J}_s,f,n) = \frac{1}{2^n}.\]

\begin{remark}
There is a notion of \textit{conditional} co-surprise analogous to conditional probability. Since $S^{co}$ is monotonic, if $\HH,\mathcal{J}\in 2^{2^X}$ then
\[S^{co}_{\mathcal{J}}(\HH) = \frac{S^{co}(\HH\cap\mathcal{J})}{S^{co}(\mathcal{J})}\]
is a well-defined function on $2^{2^X}$ with range $[0,1]$ whenever $S^{co}(\mathcal{J}) \neq 0$.
\end{remark}

Thus, while $2^{2^{X}}$ has large cardinality, $S^{co}(-,f,n)$ is only a probability measure on a subalgebra of size $2^n$, identifiable with the uniform measure on $2^n$.
\end{remark}

\autoref{prop:surprise_ineq} implies that

\begin{propo}
Suppose that $(\HH,f,n)$ is such that $S(\HH,f,n) \geq 1-\epsilon$ for $0< \epsilon <\frac{1}{2}$. Then $(\HH,f,n)$ is severely surprising at level $\epsilon < \frac{1}{2}$.  

In particular, if $\epsilon < \frac{1}{2}$ at most one of $\HH,\HH^c$ is severely surprising at level $\epsilon$.
\end{propo} 
\begin{proof}
Immediate from the above inequality on surprise.
\end{proof}

On the other hand, it is possible for \textit{neither} $\HH$ nor $\HH^c$ to be severely surprising along sample $f$ by observation $n$.

\begin{propo}
Let $\HH\subset 2^{\omega}$ be dense-codense. Then for all $f$ and $n$,
\[S(\HH,f,n) = S(\HH^c,f,n) = 0. \qedhere\]
\end{propo}
\begin{proof}
We show the argument for $\HH$ assuming density; the case of $\HH^c$ follows by codensity of $\HH$. Since $\HH$ is dense, $\HH{\upharpoonright}_{f([n])} = 2^{f([n])}$, so 
\[S(\HH,f,n) = 1- \frac{|2^{f([n])}|}{|2^{f([n])}|} = 0.\qedhere \]
\end{proof}

Nevertheless, VC finite classes of hypotheses provide us with a great wealth of severely surprising hypotheses:

\begin{propo}
Let $\HH\subset 2^{X}$ be VC finite. Then for every  $\frac{1}{2}<\epsilon<1$ there exists an $n = n(\epsilon)$ such that for every injective sample $f:\omega \to X$, if $x\in 2^{f([n])} \cap \HH$ then $\HH$ is severely surprising at level $\epsilon$.
\end{propo}
\begin{proof} 
By \autoref{prop:short-wait}, there exists an $n = n(\epsilon)$ such that $S(\HH,f,m) > 1-\epsilon$ for all injective $f:\omega\to X$ and for all $m>n$. 

Thus, it remains to show that $S(\HH,f,m) > S(\HH^c,f,m)$ for all $m>n$. There are two ways to see this. First, and most directly, by \autoref{prop:surprise_ineq} $S(\HH,f,m) > 1-\epsilon$ implies that $S(\HH^c,f,m) < \epsilon$. 

We can obtain better bounds, however, by the VC finiteness of $\HH$. Since $\HH$ is VC finite, without loss of generality we may assume $n$ is taken to be sufficiently large so that $|\HH{\upharpoonright}_{f([m])}| = p(m)$ for some polynomial, with $|\HH^c{\upharpoonright}_{f([m])}| \geq 2^m - p(m) \approx 2^m$. Thus not only is, $S(\HH^c,f,m) < S(\HH,f,m)$ for all $m>n$, we have in fact that 

\begin{equation*}
\begin{split}
\frac{S(\HH^c,f,m)}{S(\HH,f,m)} & =  \frac{1 - \frac{2^m-\tau_{\HH}(m) + |\HH{\upharpoonright}_{f([n])} \cap |\HH^c{\upharpoonright}_{f([n])}|}{2^m}}{1-\frac{\tau_{\HH}(m)}{2^m}} \\
& = \frac{\tau_{\HH}(m) - |\HH{\upharpoonright}_{f([n])} \cap \HH^c{\upharpoonright}_{f([n])}|}{2^m-\tau_{\HH}(m)} \\
& = O\left(\frac{\tau_{\HH}(m)}{2^m- \tau_{\HH}(m)}\right) \\
& = O\left(\frac{ m^{deg(\tau_{\HH})}}{2^m} \right)
\end{split}
\end{equation*}
noting that since
\[|\HH{\upharpoonright}_{f([n])} \cap \HH^c{\upharpoonright}_{f([n])}| \leq |\HH{\upharpoonright}_{f([n])}| = \tau_{\HH}(m) \]
we have 
\[\tau_{\HH}(m) - |\HH{\upharpoonright}_{f([n])} \cap \HH^c{\upharpoonright}_{f([n])}| \leq \tau_{\HH}(m). \qedhere \]
\end{proof} 

So, if $\HH$ is VC finite then not only is $\HH$ severely surprising to level $\epsilon$ if the data $x$ is compatible with $\HH$, but the ratio between $S(\HH,f,m)$ and $S(\HH^c,f,m)$ shrinks at a rate of $O\left(\frac{ m^{deg(\tau_{\HH})}}{2^m} \right)$.

For an explicit example, the hypothesis of $\HH$ that ``all coin flips are heads'' is VC finite. If $\HH$ is true, $\HH$ will be severely surprising to level $\epsilon$ so long as the number of flips $n$ satisfies
\[n > \log(\epsilon^{-1}),\]
and since $|\HH^c{\upharpoonright}_{f([n])}| = 2^n$ we have that
\[S(\HH^c,f,n) = 0\]
for all $f$ and $n$.

As in our discussion of the Formal Learning Theoretic account of falsification, VC finite classes emerge as a distinguished class of highly-testable hypotheses.

\section{Conclusion}

Over the course of this paper, the central importance of the notion of \textit{shattering} became clear:
\begin{enumerate} 
\item the nowhere density of a hypothesis can be defined in terms of shattering,
\item The VC finite classes are precisely the nowhere dense classes closed under elementary equivalence, and
\item The VC finite classes are uniquely suited for severe testability, as viewed through the lens of severe surprise.
\end{enumerate} 

It is no doubt that Vapnik himself---one of the originators of VC dimension---would not be surprised by the primacy of the notion of shattering. While he phrased his results primarily in terms of the equivalent notion of uniform two-sided convergence of the $ERM$ method of learning, he gives a probabilistic analogue to the above characterizations of VC finite hypotheses as \textit{the} class of effectively falsifiable hypotheses, writing that
\begin{quote}
    if for some some [hypothesis $\HH$] conditions of uniform convergence do not hold, the situation of nonfalsifiability will arise.\autocite[page 49]{vapnik2013nature}
\end{quote}

It is my hope that this chapter has bridged the gap between the Vapnikian account of probabilistic falsification as the study of VC finite classes with the combinatorial, logical, and topological accounts of falsification we have heretofore discussed.

\printbibliography

@article{mayo2010introduction,
  title={Introduction and Background},
  author={Mayo, Deborah G and Spanos, Aris},
  journal={Error and Inference: Recent Exchanges on Experimental Reasoning, Reliability, and the Objectivity and Rationality of Science},
  pages={28},
  year={2010},
  publisher={Cambridge University Press}
}

@article{ciesielski1995topologies,
  title={Topologies Making a Given Ideal Nowhere Dense or Meager},
  author={Ciesielski, Krzysztof and Jasinski, Jakub},
  journal={Topology and its Applications},
  volume={63},
  number={3},
  pages={277--298},
  year={1995},
  publisher={Elsevier}
}

@book{kelly1996logic,
  title={The Logic of Reliable Inquiry},
  author={Kelly, Kevin T},
  year={1996},
  publisher={Oxford University Press}
}

@incollection{mayo2005evidence,
  author      = {Mayo, Deborah G},
  title       = {Evidence as Passing Severe Tests: Highly Probable Versus Highly Probed Hypotheses},
  editor      ={Achinstein, Peter},
  booktitle   = {Scientific Evidence: Philosophical Theories and Applications},
  publisher   = {JHU Press},
  year        = {2005},
  pages       = {95-127},
  chapter     = {6}
}

@book{mayo2018severe,
  author      = {Mayo, Deborah G},
  title       = {Statistical Inference as Severe Testing},
  publisher   = {Cambridge University Press},
  year        = {2018}
}

@book{ebbinghaus2013ml,
  title={Mathematical Logic},
  author={Ebbinghaus, H-D and Flum, J{\"o}rg and Thomas, Wolfgang},
  year={2013},
  publisher={Springer Science \& Business Media}
}

@article{schulte1996epistemology,
  title={Epistemology, Reliable Inquiry and Topology},
  author={Schulte, Oliver and Juhl, Cory},
  year={1996}
}

@incollection{seldin2013relations,
  title={On the Relations and Differences Between Popper Dimension, Exclusion Dimension and VC-Dimension},
  author={Seldin, Yevgeny and Sch{\"o}lkopf, Bernhard},
  booktitle={Empirical Inference},
  pages={53--57},
  year={2013},
  publisher={Springer}
}

@book{vapnik2013nature,
  title={The Nature of Statistical Learning Theory},
  author={Vapnik, Vladimir},
  year={2013},
  publisher={Springer science \& business media}
}

@book{wasserman2006all,
  title={All of Statistics},
  author={Wasserman, Larry},
  year={2004},
  publisher={Springer Science \& Business Media}
}

\end{document}